\documentclass[amstex,11pt, leqno]{article}
\usepackage{amsmath,amsfonts,amsthm,latexsym,amssymb,mathrsfs}
\def\H_0{\mathcal{H}_0(T)}

\def\ST{\tau_{ST}}

\def\H{{\cal H}}

\newtheorem{df}{Definition}[section]
\newtheorem{thm}[df]{Theorem}

\newtheorem{rema}[df] {Remark}

\def\sfstp{{\hskip-1em}{\bf.}{\hskip1em}}

\def\enddemo{\qed \endtrivlist}
\expandafter\let\csname enddemo*\endcsname=\enddemo

\def\qedsymbol{\ifmmode\bgroup\else$\bgroup\aftergroup$\fi
  \vcenter{\hrule\hbox{\vrule
height.6em\kern.6em\vrule}\hrule}\egroup}
\def\qed{\ifmmode\else\unskip\nobreak\fi\quad\qedsymbol}

\begin{document}
\title
{ \bf The Drazin spectrum  of tensor product of \\Banach  algebra elements and elementary operators\/}

\author {\normalsize Enrico  Boasso\\ }


\date{ }


\maketitle \thispagestyle{empty} 


\vskip 1truecm

\setlength{\baselineskip}{12pt}

\begin{abstract}\noindent Given unital Banach algebras $A$ and $B$ and elements $a\in A$ and $b\in B$, 
the Drazin spectrun of $a\otimes b\in A\overline{\otimes} B$ will be fully characterized, where  
$A\overline{\otimes} B$ is a Banach algebra that is the completion of $A\otimes B$ with respect to a uniform crossnorm. 
To this end, however, first the isolated points of the spectrum
of $a\otimes b\in A\overline{\otimes} B$ need to be characterized. 
On the other hand, given Banach spaces $X$ and $Y$ and
Banach space operators $S\in L(X)$ and $T\in L(Y)$, using similar arguments
the Drazin spectrum of  $\ST\in L( L(Y,X))$, the elementary operator defined by $S$ and $T$,
will be fully characterized. \par
\par
\noindent $\bf{Keywords\colon}$ \rm Drazin spectrum;  Banach algebra; tensor product; 
Banach space; elementary operator. \par
\noindent $\bf{AMS }$  $\bf{Subject}$ $\bf{Classification\colon}$ \rm 47A10; 46H05, 47B49.
\end{abstract}


\section {\sfstp Introduction}\setcounter{df}{0}
\
\indent The relationships among tensor products and spectral theory have been extensively studied. For example,
it is well known that  given unital Banach algebras $A$ and $B$, if $A\overline{\otimes}B$ is a Banach algebra that
is the completion of the algebraic tensor product $A\otimes B$ with respect to a uniform crossnorm,
then $\sigma (a\otimes b)=\sigma (a)\sigma (b)$, where $a\in A$, $b\in B$ and $\sigma (\cdot )$ is the
usual spectrum. When $X$ and $Y$ are Banach spaces, $X\overline{\otimes}Y$ is a
completion of the algebraic tensor product $X\otimes Y$ with respect to an appropriate norm and $S\colon X\to X$
and $T\colon Y\to Y$ are Banach space operators, then $\sigma_*(S\otimes T) =\sigma_* (S)\sigma_*(T)$,
where $\sigma_*(\cdot)$ is either the approximate point spectrum or the defect spectrum (\cite{I2}).
However, if $\sigma_+(\cdot)$ denotes either the Fredholm or the Browder spectrum, then the
following formula holds: $\sigma_+(S\otimes T)=\sigma (S)\sigma_+(T)\cup \sigma_+(S)\sigma (T)$
(\cite{E, CD}). Formulas similar to the last one were proved also for other spectra such as the 
upper and lower Fredholm spectra and the Browder approximate point spectrum, if instead of the usual
spectrum, the approximate point spectrum is considered (\cite{I2, DDK}). On the other hand,
it is well known that the spectral theories of tensor products of operators and elementary operators
are deeply connected (\cite{BP, E}). Now well, similar formulas for the aforementioned spectra
of elementary operators were proved in \cite{E, BDJ}.

\indent The first objective of this article is to fully characterize $\sigma_{DR} (a\otimes b)$,
the Drazin spectrum of $a\otimes b\in A\overline{\otimes}B$, in terms of the spectrum and the Drazin
spectrum of $a\in A$ and $b\in B$, where $A$, $B$ and $A\overline{\otimes}B$ are as at the
beginning of the previous paragraph. In fact, in section 4  the relationship between  $\sigma_{DR} (a\otimes b)$
and $\mathbb{D}= \sigma (a) \sigma_{DR}(b)\cup \sigma_{DR} (a)\sigma(b)$ will be studied. In particular, it will prove that
 $\sigma_{DR} (a\otimes b)\subseteq \mathbb{D}$ and $\sigma_{DR} (a\otimes b)\setminus\{0\}= \mathbb{D}\setminus\{0\}$, necessary and sufficient conditions
to characterize when $\sigma_{DR} (a\otimes b)$  and $\mathbb{D}$ coincide will be given  and clearly, 
when $\sigma_{DR} (a\otimes b)\subsetneq\mathbb{D}$, then  $\mathbb{D}=\sigma_{DR} (a\otimes b)\cup \{0\}$,
$0\notin \sigma_{DR} (a\otimes b)$. On the other hand, the second objective of this work is to study the Drazin spectrum of
 elementary  operators between Banach spaces, which will be fully characterized in section 5 using arguments similar
to the ones  in  section 4.\par

\indent However, to this end, the isolated points of $\sigma (a\otimes b)$, $a\otimes b\in A\overline{\otimes}B$,
$A$, $B$ and $A\overline{\otimes}B$ as before, need to be characterized in terms of the poles and the complements of the poles in the isolated points of the 
spectrum of both $a\in A$ and 
$b\in B$. This will be done in section 3, after having recalled in section 2 the preliminary definitions and results that
will be used in the present work. In addition, this characterization will deepen the knowledge of the
isolated points of the spectrum of elementary tensors obtained in \cite{HK, DHK}. \par

\indent It is worth noticing that the Drazin spectrum is a key notion in the research area of
Weyl's and Browder's theorems and their generalizations. What is more,
in the recent past (generalized) Weyl's and (generalized) Browder's theorems of both
tensor product operators and elementary operators have been studied (\cite{SK, AD,KD, Du, DDK, BDJ, HK, DHK}). On the other hand, 
the set of isolated points of the spectrum is another central notion in this area of research. Therefore, 
the results obtained in the present work, apart from their interest related to spectral theory and the Drazin inverse,
also present an interest related to the area of
Weyl's and Browder's theorems.

\section {\sfstp Preliminary definitions and results}\setcounter{df}{0}
\
\indent From now on $A$ will denote a unital Banach algebra with unit $e$.
 Recall that the element $a\in A$ is said to be 
\it Drazin invertible\rm, if there exists a necessarily unique $b\in A$ and some $m\in \Bbb N$ such that
$$
a^mba=a^m, \hskip.5truecm  bab=b,\hskip.5truecm ab=ba.
$$
The least non-negative integer $m$ for which the above equations holds is called the \it index \rm of $a$.
In addition, if the Drazin inverse of $a$ exists, then it will be denoted by $a^D$; concerning this research area
see for example \cite{D, K, RS, BS, Bo}. \par

\indent On the other hand, the notion of \it regularity \rm was introduced and studied in \cite{KM, MM}. Recall that
given a unital Banach algebra $A$ and a regularity $\mathcal R\subseteq A$,
the \it spectrum derived from the 
regularity \rm $ \mathcal R$  is defined as $\sigma_{\mathcal R} (a)=\{ \lambda\in \Bbb C\colon a-\lambda\notin  \mathcal R\}$,
where $a\in A$ and $ a-\lambda$ stands for  $a-\lambda e$, see \cite{KM}. In addition, the \it resolvent set of $a$ defined by the regularity $\mathcal R$ \rm
is the set $\rho_{\mathcal R} (a)=\{ \lambda\in \Bbb C\colon a-\lambda\in  \mathcal R\}$. Naturally,
when $\mathcal R=A^{-1}$, the set of all invertible elements of $A$, $\sigma_{A^{-1}}(a)=\sigma (a)$, the spectrum
of $a$, and $\rho_{A^{-1}}(a) =\rho(a)=\mathbb C\setminus\sigma (a)$, the resolvent set of $a$. Next consider the set
$\mathcal{DR} (A)$ = $\{ a\in A\colon  \hbox{  a is Drazin invertible}\}$. According to
\cite[Theorem 2.3]{BS}, $\mathcal{DR} (A)$ is a regularity. This fact led to the following definition,
see \cite{BS}.\par

\begin{df}\label{def1}
Let $A$ be a unital Banach algebra. The Drazin spectrum
of  $a\in A$ is the set
$$
\sigma_{ \mathcal{ DR} }(a) = \{\lambda\in \Bbb C \colon a-\lambda\notin  \mathcal{ DR}(A)\}.
$$
\end{df}
 
\indent Naturally $\sigma_{ \mathcal{ DR} }(a)\subseteq \sigma(a)$, $\rho (a)\subseteq\rho_{\mathcal{ DR}}(a)$
and according to \cite[Theorem 1.4]{KM}, the Drazin spectrum
of a Banach algebra element satisfies the spectral mapping theorem for analytic functions
defined on a neighbourhood of the ususal spectrum which are non-constant on each component of its 
domain of definitioin (see also \cite[Corollary 2.4]{BS}). In addition, according to \cite[Proposition 2.5]{BS},
$\sigma_{ \mathcal{ DR} }(a)$ is a closed subset of $\Bbb C$. Next a description of the  spectrum and the Drazin spectrum 
in terms of the isolated and limit points of the spectrum given in 
\cite{Bo} will be recalled.  However, to this end, some preparation
is needed. First of all, if $K\subseteq \mathbb C$, then, iso $K$ will stand for the isolated points of
$K$ and acc $K= K\setminus$ iso $K$.\par

\indent Let $A$ be a unital Banach algebra and $a\in A$. As in the case of   
the spectrum of a Banach space operator, the resolvent function of 
$a$, $R(\cdot, a)\colon\rho(a)\to A$, is holomorphic and iso $ \sigma(a)$ coincides with the set of
isolated singularities of $R(\cdot ,  a)$. Furthermore, as in the case of an operator,
see \cite[p. 305]{T}, if $\lambda_0\in$ iso $\sigma(a)$, then it is possible to
consider the Laurent expansion of  $R(\cdot ,  a)$ in terms of $(\lambda-\lambda_0)$.
In fact,
$$
R(\lambda ,  a)=\sum_{n\ge 0}a_n(\lambda-\lambda_0)^n +\sum_{n\ge 1}b_n(\lambda-\lambda_0)^{-n},
$$
where $a_n$ and $b_n$ belong to $A$ and are obtained in an standard way using the
functional calculus. In addition, this representation is valid when $0<\mid \lambda-\lambda_0\mid<\delta$, for 
any $\delta$ such that $\sigma(a)\setminus \{\lambda_0\}$ lies outside the circle $\mid\lambda-\lambda_0\mid=\delta$. 
What is more important, the discussion of \cite[pp. 305 and 306]{T}
can be repeated for elements in a unital Banach algebra. Consequently, $\lambda_0$
will be called \it a pole of order $p$ of   $R(\cdot ,  a)$, \rm if there is $p\ge 1$ such that
$b_p\neq 0$ and $b_m=0$, for all $m\ge p+1$. The set of poles of $a$
will be denoted by $\Pi(a)$; clearly $\Pi (a)\subseteq $ iso $\sigma (a)$. \par

\indent Let $X$ be a Banach space and denote by $L(X)$ the algebra of bounded and linear maps
defined on and with values in $X$. If $T\in L(X)$, then $N(T)$ and $R(T)$ will stand for the null space and the
range of $T$ respectively. Recall that the \it descent \rm and the \it ascent \rm of $T\in L(X)$ are
$d(T) =\hbox{ inf}\{ n\ge 0\colon  R(T^n)=R(T^{n+1})\}$ and
$ a(T)=\hbox{ inf}\{ n\ge 0\colon  N(T^n)=N(T^{n+1})\}$
respectively, where if some of the above sets is empty, its infimum is then defined as $\infty$, see  
for example \cite{MM, BS}. Now well, according to \cite[Theorem 4]{K}, $\Pi(T)=\{\lambda\in\sigma (a)\colon
a(T-\lambda)\hbox{ and } d(T-\lambda) \hbox{ are finite}\}$. What is more, if $A$ is a unital Banach algebra and $a\in A$, 
then, according to  \cite[Theorem 11]{Bo}, $\Pi (a)=\Pi (L_a)=\Pi (R_a)$,
where $L_a$, $R_a\in L(A)$ are the operators defined by left and right multiplication,
i.e., given $x\in A$, $L_a(x)=ax$ and $R_a(x)=xa$ respectively.\par
\indent Now well, according to \cite[Theorem 12]{Bo}, 
$$
 \sigma(a)=\sigma_{\mathcal{DR}}(a)\cup \Pi(a), \hskip.2truecm
\sigma_{\mathcal{DR}}(a)= \hbox{ \rm acc }\sigma (a)\cup I(a), \hskip.2truecm \Pi(a)=\sigma(a)\cap\rho_{\mathcal{DR}}(a),
$$
where $I(a)=$ iso $\sigma (A)\setminus \Pi (a)$. Note that $\sigma_{\mathcal{DR}}(a)\cap \Pi(a)=\emptyset$
and (acc $\sigma (a))\cap I(a)=\emptyset$.\par

\indent Finally, if $A$ and $B$ are unital Banach algebras, then $A\otimes B$ will stand for the 
algebraic tensor porduct of $A$ and $B$, while $A\overline{\otimes} B$ will denote a Banach algebra that is  the completion of
$A\otimes B$ with respect to a uniform crossnorm; see for example \cite[Chapter 11, Section 7]{H}.

\section {\sfstp Isolated points}\setcounter{df}{0}
\
\indent To prove the key results of this article, the isolated points of the spectrum of the tensor product of two
Banach algebra elements need to be characterized first. To this end, however, some prepartion is needed.\par

\begin{rema}\label{rem11}\rm Let $A$ be a unital Banach algebra and consider $a\in A$. Recall that
if $p^2=p\in A$, then $pAp$ is a unital Banach algebra with unit $p$. \par
\noindent (i) Necessary and sufficient for $\lambda\in \sigma (a)$ to belong to $\Pi (a)$ is that
there exists $0\neq p^2=p\in A$ such that $ap=pa$, $(a-\lambda)p$ is nilpotent and $(a-\lambda+p)\in A^{-1}$.
Note in particular that $\sigma (a)=\Pi (a)=\{\lambda\}$ if and only if $p=e$, which is equivalent to the fact that $a-\lambda$
is nilpotent.\par
\noindent (ii) The number $\lambda\in \sigma (a)$ belongs to $I(a)$ if and only if
there is $0\neq p^2=p\in A$ such that $ap=pa$, $(a-\lambda)p$ is quasi-nilpotent but not nilpotent and $(a-\lambda+p)\in A^{-1}$.
What is more, as in the previous statement, $\sigma (a)=I (a)=\{\lambda\}$ if and only if $p=e$, which is equivalent to the fact that
$a-\lambda$ is quasi-nilpotent but not nilpotent.\par
\noindent Statements (i) and (ii) are well known and they can be derived, for expample, from \cite{RS} and \cite{Ko}. Furthermore,
it is not difficult to prove that statement (i) (respectively statement (ii)) is equivalent to the existence
of $0\neq p^2=p\in A$ such that $ap=pa$ (equivalently, $a-\lambda= p(a-\lambda)p + p'(a-\lambda)p'$, $p'=e-p$), 
$p(a-\lambda)p$ is nilpotent (respectively quasi-nilpotent but not nipotent) in the Banach algebra $pAp$ and $p'(a-\lambda)p'$ is invertible in the Banach
algebra $p'Ap'$. Next consider $B$ another unital Banach algebra and $b\in B$. Note that both the 
identity of $B$ and of $A\overline{\otimes} B$ will be denoted by $e$.\par

\noindent (iii) Since $\sigma (a\otimes b)=\sigma(a)\sigma(b)$ (\cite[Theorem 11.7.6]{H}),
according to \cite[Theorem 6]{HK},
$$
\hbox{iso }\sigma(a\otimes b)\setminus \{ 0\}= 
(\hbox{iso } \sigma(a)\setminus \{ 0\})( \hbox{iso } \sigma(b)\setminus \{ 0\}).
$$
\pagestyle{myheadings} \markboth{  }{ \hskip7truecm \rm Enrico Boasso  }
\noindent (iv) Set
$$
\mathbb L= (I(a)\setminus \{ 0\}) (I(b)\setminus \{ 0\})\cup  (I(a)\setminus \{ 0\}) (\Pi (b)\setminus \{ 0\})
\cup  (\Pi (a)\setminus \{ 0\}) (I(b)\setminus \{ 0\}).
$$
Then clearly, iso $\sigma(a\otimes b)\setminus \{ 0\}= \mathbb L\cup (\Pi (a)\setminus \{ 0\})(\Pi
(b)\setminus \{ 0\})$.\par

\noindent (v) Let $\lambda\in$ iso
$\sigma (a\otimes b )\setminus\{0\}$. Then, it is not difficult to prove that there exist $n\in\mathbb N$ and
finite sequences $\{\mu\}=\{\mu_1 ,\dots ,\mu_n\}\subseteq$ iso $\sigma(a)\setminus \{ 0\}$
and $\{\nu\}=\{\nu_1,\dots ,\nu_n\}\subseteq$ iso
$\sigma(b)\setminus \{ 0\}$ such that $\lambda=\mu_i \nu_i$ for all
$i=1,\dots , n$, and that if there exist $\mu\in$  $\sigma(a)\setminus \{ 0\}$
and  $\nu\in$ $\sigma(b)\setminus \{ 0\}$ such that $\lambda=\mu\nu$, then there is 
$i_0$, $1\le i_0\le n$, for which $\mu=\mu_{i_0}$ and $\nu=\nu_{i_0}$.\par

\end{rema}

\indent In the following theorem the position of $0$ in the isolated and the limit points of the
$\sigma (a\otimes b)$   will be studied, where $a\in A$ and $b\in B$, $A$ and $B$ unital Banach algebras. 
To this end, the description of the spectrum and the Drazin spectrum presented in \cite[Theorem 12]{Bo} will
be central. Note that if $0\in $ iso $\sigma (a\otimes b)$, then $0\in $ iso $\sigma (a)$ or   $0\in $ iso $\sigma (b)$.

\pagestyle{myheadings} \markboth{  }{ \hskip7truecm \rm Enrico Boasso  }

\begin{thm}\label{thm3}Let $A$ and $B$ be two unital Banach algebras and consider
$a\in A$ and $b\in B$. Then, the following statements hold.\par
\noindent \rm (i)\it $\sigma (a)=\Pi (a)=\{0\}$ or $\sigma (b)=\Pi (b)=\{0\}$ if and only if $\sigma (a\otimes b)=\Pi (a\otimes b)=\{0\}$.\par
\noindent \rm (ii)\it Suppose that neither $a$ nor $b$ is  nilpotent. Then necessary and sufficient for $\sigma (a)=I(a)=\{0\}$  or $\sigma (b)=I(b)=\{0\}$  is that $\sigma (a\otimes b)=I(a\otimes b)=\{0\}$.\par
\noindent \rm (iii)\it If $0\in \Pi (a)$ $($respectively $0\in \Pi (b)$$)$ and $b$ is invertible $($respectively $a$ is invertible$)$
or if  $0\in \Pi (a)\cap \Pi (b)$, then $0\in \Pi (a\otimes b)$.\par
\noindent \rm (iv)\it Suppose that $a$ is not nilpotent $($respectively $b$ is not nilpotent$)$. Then, 
if $0\in\Pi (a)$ $($respectively $0\in \Pi (b)$$)$ and $0\in I(b)$ $($respectively $0\in I(a)$$)$, then
$0\in I(a\otimes b)$.\par
\noindent \rm (v) \it If $b$ $($respectively $a)$ is invertible and $0\in I(a)$ $($respectively $0\in I(b))$, then $0\in I(a\otimes b)$.  \par
\noindent \rm (vi)\it If $0\in I (a)\cap I(b)$, then $0\in I(a\otimes b)$. \par
\noindent \rm (vii)\it Suppose that $ \sigma (b)\neq \{0\}$ $($respectively  $ \sigma (a)\neq \{0\}$$)$. Then, if $0\in$ acc $\sigma (a)$
$($respectively $0\in$ acc $\sigma (b)$$)$, then $0\in $ acc $\sigma (a\otimes b)$.\par
\end{thm}
\begin{proof}(i). If $a\in A$ or $b\in B$ is nilpotent, then clearly $a\otimes b\in A\overline{\otimes }B$ is nilpotent.
On the other hand, if $a\otimes b\in A\overline{\otimes }B$ is nilpotent, since the norm considered on $A\overline{\otimes }B$
is a crossnorm, then $a$ or $b$ must b nilpotent.\par

\noindent (ii). If $\sigma (a)=I(a)=\{ 0\}$  or $\sigma (b)=I(b)=\{ 0\}$, then $\sigma (a\otimes b)=\{ 0\}$. Now well, if $\sigma (a\otimes b)=\Pi (a\otimes b)=\{0\}$,
then according to statement (i), $a$ or $b$ must be nilpotent, which is impossile. Hence, $\sigma (a\otimes b)= I (a\otimes b)=\{0\}$.
To prove the converse implication, note that $\sigma (a)=\{0\}$ or $\sigma (b)=\{0\}$. However, since $a$ and $b$ are not nilpotent,
$\sigma (a)=I(a)=\{0\}$  or $\sigma (b)=I(b)=\{0\}$.\par

\noindent (iii). The conditions of the statement under consideration imply that $a\in A$ and $b\in B$ are Drazin invertible.
Using the Drazin inverses of $a\in A$ and $b\in B$, it is not difficult to prove that 
$a\otimes b\in  A\overline{\otimes }B$ is Drazin invertible. \par 

\noindent (iv) Suppose that $0\in \Pi (a) \cap I(b)$. Clearly $0\in$ iso $\sigma (a\otimes b)$. Accoding to Remark \ref{rem11}(i)-(ii), there exist $0\neq p^2=p\in A$ and $0\neq q^2=q\in B$
such that $ap=pa$, $bq=qb$, $(e-p)a(e-p)$ is invertibe in $(e-p)A(e-p)$, $(e-q)b(e-q)$ is invertible in $(e-q)B(e-q)$,  $pap\in pAp$ is nilpotent
and $pbq\in qBq$ is quasi-nilpotent but not nilpotent. Let $r=p\otimes q+(e-p)\otimes q+p\otimes (e-q)\in A\overline{\otimes} B$. Note that
$r=r^2$ and $e-r=(e-p)\otimes (e-q)$. In particular, $r\neq 0$ and $ra\otimes b=a\otimes br$. In addition, $(e-r)a\otimes b(e-r)=(e-p)a(e-p)\otimes (e-q)b(e-q)$,
which is is invertible in $(e-p)A(e-p)\overline{\otimes }(e-q)B(e-q)=(e-r)A\overline{\otimes}B(e-r)$. Moreover, a straightforward calculation proves that
$ra\otimes br=s_1+s_2+s_3$, where $s_1=pap\otimes qbq\in S_1=pAp\overline{\otimes} qBq$, $s_2=(e-p)a(e-p)\otimes qbq\in S_2=(e-p)A(e-p)\overline{\otimes} qBq$
and $s_3=pap\otimes (e-q)b(e-q)\in S_3=pAp\overline{\otimes} (e-q)B(e-q)$. Note that since $s_is_j=0$, $1\le i\neq j\le 3$, 
$(ra\otimes br)^n=s_1^n+s_2^n+s_3^n$.\par
\indent Now well, since $0\in \Pi(a)\cap I(b)$,  $pap$ is nilpotent and $qbq$ is quasi-nilpotent but not nilpotent.
In particular, there is $k\in \mathbb N$ such that for $n\ge k$, $(ra\otimes br)^n=s_2^n$. However, since
$(e-p)a(e-p)\in (e-p)A(e-p)$ is invertible,  $qbq\in  qBq$ is quasi-nilpotent but not nilpotent and the norm 
on the tensor product is a crossnorm,  
$ ra\otimes b r$ is quasi-nilpotent but not nilpotent. Therefore, according to Remark \ref{rem11}(ii),
$0\in I(a\otimes b)$. The case  $0\in I(a)\cap \Pi(b)$ can be proved in a similar way.\par
\noindent (v) Adapt the argument in the proof of statement (iv) using in particular $p=0$ (respectively $q=0$)
in the case $a\in A$ invertible (respectively $b\in B$ invertible).\par

\noindent (vi) Proceed as in the proof of statement (iv) but assuming that $pap\in pAp$ is quasi-nilpotent but not nilpotent.
Note then that since $(ra\otimes br)^n=s_1^n+s_2^n+s_3^n$ and  $s_is_j=0$, $1\le i\neq j\le 3$, necessary and suficient for $ra\otimes br$ to be nilpotent
is that there exists $k\in \mathbb N$ such that $s_1^k=s_2^k=s_3^k=0$, which, since the norm of $A\overline{\otimes }B$
is a cross norm and $pap\in pAp$ and $qbq\in qBq$ are quasi-nilpotent but not nilpotent, is impossible.\par

\noindent (vii) A straightforward calculation proves this statement.
\end{proof}

\indent To characterize the non-null isolated points of $\sigma (a\otimes b)$ ($a\in A$, $b\in B$, $A$ and $B$ unital Banach algebras), 
first of all a particular case need to  be considered.\par

\begin{thm}\label{thm6}Let $A$ and $B$ be two unital Banach algebras and consider
$a\in A$ and $b\in B$ such that $\sigma (a)=\{\mu\}$, $\sigma (b)=\{\nu\}$ and  $\lambda=\mu\nu\neq 0$.\par
\noindent \rm (i)\it  If $\sigma (a)=\Pi(a)=\{\mu\}$ and $\sigma (b)=\Pi (b)=\{\nu\}$, then  $\sigma (a\otimes b)=\Pi (a\otimes b)=\{\lambda\}$.\par
\noindent \rm (ii)\it If either $\sigma (a)=I(a)=\{\mu\}$ or $\sigma (b)=I (b)=\{\nu\}$, then  $\sigma (a\otimes b)=I (a\otimes b)=\{\lambda\}$.\par
\end{thm}
\pagestyle{myheadings} \markboth{  }{ \hskip7truecm \rm Enrico Boasso  }
\begin{proof} First of all, note that $\sigma (a\otimes b)=\{\lambda\}$.\par
\noindent (i). Since $a\otimes b-\lambda= (a-\mu)\otimes b+\mu \otimes (b-\nu)$, $a-\mu\in A$ and $b-\nu\in B$
are nilpotent and  $(a-\mu)\otimes b$ and $\mu \otimes (b-\nu)$ commute, $a\otimes b-\lambda$ is nilpotent.\par
\noindent (ii) Suppose that $b-\nu\in B$ is quasi-nilpotent but not nilpotent (the case $a-\mu\in A$ is
quasi-nilpotent but not nilpotent can be proved using similar argments).  Let $c=a-\mu$, $d=\mu (b-\nu)$
and note that $a\otimes b-\lambda=c\otimes b+e\otimes d$ and $c\otimes b$ and $e\otimes d$ commute. 
Now well, since $c\in A$ is quasi-nilpotent, according to \cite[Theorem 7.4.2]{H}, there is a sequence
$(h_n)_{n\in\mathbb N}\subseteq A$ such that $\parallel h_n\parallel =1$ and $(ch_n)_{n\in\mathbb N}\subseteq A$
converges to $0\in A$. In addition, since $d\in B$ is quasi-nilpotent but not nilpotent, for each 
$k\in \mathbb N$, $d^k\neq 0$ and $l_k=\frac{2e}{\parallel d^k\parallel}$ is such that $\parallel d^kl_k\parallel=2$.
Next consider $z_{k,j}=\frac{k!}{j!(k-j)!}$, where $1\le j\le k\in\mathbb N$. Then, an easy calculation proves that given 
$k\in\mathbb N$, there is $n_k\in \mathbb N$ such that 
$$
\parallel (\sum_{j=1}^k z_{j,k}c^j\otimes b^jd^{k-j})h_{n_k}\otimes l_k\parallel<1.
$$
In particular,
$$
\parallel (a\otimes b-\lambda)^kh_{n_k}\otimes l_k\parallel=\parallel (c\otimes b+e\otimes d)^kh_{n_k}\otimes l_k\parallel\ge
\parallel h_{n_k}\otimes d^kl_k\parallel -\parallel (\sum_{j=1}^k z_{j,k}c^j\otimes b^jd^{k-j})h_{n_k}\otimes l_k\parallel>1.
$$
Therefore, $a\otimes b-\lambda\in A\overline{\otimes } B$ is quasi-nilpotent but not nilpotent, equivalently
$\lambda\in I(a\otimes b)$.
\end{proof}

\indent In the following theorem, the non-null isolated points of the spectrum of the tensor product 
of two Banach algebra elements will be fully characterized. To this end, note that if $A$ is a  unital Banach algebra and $p^2=p\in A$, then
the spectrum of $z\in pA$ in the unital Banach algebra $pA$ will be denoted by
$\sigma_{pA} (z)$. In addition,  if $q^2=q\in B$, then with the norm restricted from
$A\overline{\otimes }B$, $pA\overline{\otimes} qB$ is a Banach algebra with a
uniform crossnorm.\par

\begin{thm}\label{thm1}Let $A$ and $B$ be two unital Banach algebras and consider
$a\in A$ and $b\in B$. Then, the following statements hold.\par
\noindent \rm (i) \it $\mathbb L= I(a\otimes b)\setminus \{0\}$.\par
\noindent \rm (ii) \it $\Pi (a\otimes b)\setminus \{0\}= (\Pi (a)-\{0\})(\Pi (b)-\{0\})\setminus \mathbb L$.
\end{thm}
\pagestyle{myheadings} \markboth{  }{ \hskip7truecm \rm Enrico Boasso  }
\begin{proof}Note that according to Remark \ref{rem11}(iv), statement (i) implies statement (ii).\par
\indent To prove statement (i), let $\lambda\in $ iso $\sigma (a\otimes b)\setminus\{0\}$. Then,
according to Remark \ref{rem11}(v), there exist $n\in \mathbb N$ and finite spectral sets
$\{\mu\}=\{\mu_1 \dots , \mu_n\}\subseteq $ iso $\sigma (a)$ and 
$\{\nu\}=\{\nu_1 \dots , \nu_n\}\subseteq $ iso $\sigma (b)$ such that $\lambda=\mu_i\nu_i$,
for all $1\le i\le n$
and that if there exist $\mu\in$  $\sigma(a)\setminus \{ 0\}$
and  $\nu\in$ $\sigma(b)\setminus \{ 0\}$ such that $\lambda=\mu\nu$, then there is 
$i_0$, $1\le i_0\le n$, with the property $\mu=\mu_{i_0}$ and $\nu=\nu_{i_0}$. In addition, using the functional calculus it is not difficult to prove that
there are idempotents $p$ and $(p_i)_{i=1}^n$ in $A$ and $q$ and $(q_i)_{i=1}^n$ in $B$
such that $p=\sum_{i=1}^n p_i$, $p_ip_j=0$ and $p_iap_j=0$ for $i\neq j$, $p'p_i=0$ ($p'=e-p$), $p_iap'=0$ and $p_ia=ap_i=p_iap_i$, $1\le i\le n$,
$pap'=p'ap=0$,
$A=pA\oplus p'A$, $pA=\oplus_{i=1}^n p_iA$,
$\sigma_{pA}(pa)=\{\mu\}$, $\sigma_{p_iA}(p_ia)=\sigma_{pA}(p_ia)=\sigma(p_ia)=\{\mu_i\}$,
$\sigma(p'a)_{p'A}=\sigma (a)\setminus\{\mu\}$ and 
 $q=\sum_{i=1}^n q_i$, $q_iq_j=0$ and $q_ibq_j=0$ for $i\neq j$, $q'q_i=0$ ($q'=e-q$), $q_ibq'=0$ and $q_ib=bq_i=q_ibq_i$, $1\le i\le n$, 
$B=qB\oplus q'B$, $qB=\oplus_{i=1}^n q_iB$,
$\sigma_{qB}(qb)=\{\nu\}$, $\sigma_{q_iB}(q_ib)=\sigma_{qB}(q_ib)=\sigma(q_ib)=\{\nu_i\}$,
$\sigma(q'b)_{q'B}=\sigma (b)\setminus\{\nu\}$. Next, define $r=\sum_{i=1}^n p_i\otimes q_i\in A\overline{\otimes }B$.
Then, it is not difficult to prove that $0\neq r=r^2$, $ra\otimes b=a\otimes br$ and $e-r=p\otimes q' +p'\otimes q+p'\otimes q' +\sum_{1\le i\neq j\le n} p_i\otimes q_j$.\par

\indent Note that $\lambda\notin\sigma_{pA\overline{\otimes}q'B}(p\otimes q' a\otimes b)=\sigma_{pA\overline{\otimes}q'B}(pa\otimes q'b)
=\sigma_{pA} (pa)\sigma_{q'B}(q'b)=\{\mu\}(\sigma (b)-\{\nu\})$. However, a straightforward calculation, using in
particular this latter fact, proves that $z_1=p\otimes q' a\otimes b p\otimes q'-\lambda p\otimes q'$ is invertible in $p\otimes q' A\overline{\otimes}B p\otimes q'$.
What is more, similar arguments prove that $z_2=p'\otimes q a\otimes bp'\otimes q-\lambda p'\otimes q$
is invertible in $p'\otimes q A\overline{\otimes}B p'\otimes q$, 
$z_3=p'\otimes q' a\otimes bp'\otimes q'-\lambda p'\otimes q'$
is invertible in $p'\otimes q' A\overline{\otimes}B p'\otimes q'$
and $z_{i,j}=p_i\otimes q_j a\otimes bp_i\otimes q_j-\lambda p_i\otimes q_j$
is invertible in $p_i\otimes q_j A\overline{\otimes}B p_i\otimes q_j$,
for $1\le i\neq j\le n$.   Now well, according to the properties of the
idempotents $p$, $p_i$, $q$ and $q_i$, $1\le i\le n$, recalled in the previous paragraph, it is not difficult to prove that
 $(e-r)(a\otimes b-\lambda)(e-r)=z_1+z_2+z_3+\sum_{1\le i \neq j\le n}z_{i,j}$,
$(e-r)A\overline{\otimes}B(e-r)=pAp\overline{\otimes} q'Bq'\oplus p'Ap'\overline{\otimes} qBq\oplus p'Ap'\overline{\otimes} q'Bq'
\oplus_{1\le i\neq j\le n}p_iAp_i\overline{\otimes}q_j Bq_j$ and  $(e-r)(a\otimes b-\lambda)(e-r)$ is invertible in $(e-r)A\overline{\otimes}B(e-r)$.\par

\indent Moreover, using again the properties of the aforementioned idempotents, it is possible  to prove that 
$rA\overline{\otimes}Br=\oplus_{1\le i\le n} p_iAp_i\overline{\otimes }q_iBq_i$ and $r(a\otimes b-\lambda)r=\sum_{1\le i\le n} p_i\otimes q_i(a\otimes b-\lambda)p_i\otimes q_i$.
What is more, since $p_i\otimes q_i(a\otimes b-\lambda)p_i\otimes q_i$ is quasi-nilpotent in $p_i\otimes q_iA\overline{\otimes} Bp_i\otimes q_i$ ($1\le i\le n$),
$r(a\otimes b-\lambda)r$ is quasi-nilpotent. Note that $r(a\otimes b-\lambda)r$ is nilpotent if and only if $p_i\otimes q_i(a\otimes b-\lambda)p_i\otimes q_i$
is nilpotent for all  $i=1, \dots , n$.\par

\indent Let $\lambda \in I(a\otimes b)\setminus\{0\}$. Suppose that for all $i=1, \cdots , n$, $\mu_i\in \Pi (a)$ and $\nu_i\in \Pi (b)$. 
In particular, according to Remark \ref{rem11}(i) and the properties of the idempotents $p_i$ and $q_i$, $1\le i\le n$, 
$\sigma_{p_iA } (p_ia)=\Pi(p_ia)=\{\mu_i\}$ and $\sigma_{q_iB } (q_ib)=\Pi(q_ib)=\{ \nu_i \}$.  However, according to
Theorem \ref{thm6}(i), $\sigma_{p_iA\overline{\otimes} q_iB}(p_ia\otimes q_ib)=\Pi (p_ia\otimes q_ib)=\{\mu_i\nu_i\}=\{\lambda\}$
($1\le i\le n$). As a result, $p_iap_i\otimes q_ibq_i-\lambda p_i\otimes q_i=p_ia\otimes q_ib-\lambda p_i\otimes q_i$
is nilpotent in $p_iA\overline{\otimes}q_iB$, and hence in   $p_iAp_i\overline{\otimes}q_iBq_i$ ($i=1, \dots ,n$).
Consequently, $r(a\otimes b-\lambda)r\in rA\overline{\otimes}Br$ is nilpotent. Now well, according to what has been proved, 
$(e-r)(a\otimes b-\lambda)(e-r)$ is invertible in $(e-r)A\overline{\otimes}B(e-r)$. Then, according to Remark \ref{rem11}(i) and its equivalent
formulation (see Remark \ref{rem11}), $\lambda\in \Pi (a\otimes b)$,
which  is impossible. Therefore, $ I(a\otimes b)\setminus\{0\}\subseteq\mathbb L$.\par

\indent On the other hand, if $\lambda\in \mathbb L$, then according to Remark \ref{rem11}(ii) and the properties of the idempotents $p_i$ and $q_i$, 
$1\le i\le n$, there is $i_0$, $1\le i_0\le n$, such that 
$\sigma_{p_{i_0}A } (p_{i_0}a)= I(p_{i_0}a)=\{\mu_{i_0}\}$ or $\sigma_{q_{i_0}B } (q_{i_0}b)= I (q_{i_0}b)=\{ \nu_{i_0} \}$.
Therefore, according to Thorem \ref{thm6}(ii), 
$\sigma_{p_{i_0}A\overline{\otimes} q_{i_0}B}(p_{i_0}a\otimes q_{i_0}b)=I (p_{i_0}a\otimes q_{i_0}b)=\{\mu_{i_0}\nu_{i_0}\}=\{\lambda\}$.
Then, $p_{i_0}ap_{i_0}\otimes q_{i_0}bq_{i_0}-\lambda p_{i_0}\otimes q_{i_0}=p_{i_0}a\otimes q_{i_0}b-\lambda p_{i_0}\otimes q_{i_0}$
is quasi-nilpotent but not nilpotent in $p_{i_0}A\overline{\otimes}q_{i_0}B$, and hence in   $p_{i_0}Ap_{i_0}\overline{\otimes}q_{i_0}Bq_{i_0}$.
As a result,  $r(a\otimes b-\lambda)r\in rA\overline{\otimes}Br$
is quasi-nilpotent but not nilpotent. Now well, as in the previous paragraph,  
$(e-r)(a\otimes b-\lambda)(e-r)$ is invertible in $(e-r)A\overline{\otimes}B(e-r)$. Then, according to Remark \ref{rem11}(ii) and its 
equivalent formulation (see Remark \ref{rem11}), $\lambda\in I(a\otimes b)\setminus\{0\}$. Thus, $ \mathbb L\subseteq I(a\otimes b)\setminus\{0\}$.
\end{proof}
\section {\sfstp The Drazin spectrum }\setcounter{df}{0}
\
\indent In this section, the Drazin spectrum of the tensor product of  two
Banach algebra elements will be characterized. To this end, given $A$ and $B$ two unital Banach algebras
and $a\in A$ and $b\in B$, consider the sets $\mathbb A= \sigma (a)$(acc $\sigma (b))\cup $ (acc
$\sigma (a))\sigma (b)$ and
$\mathbb B =I(a)I(b)\cup I(a)\Pi (b)\cup \Pi (a)I(b)$.
 Note that according to \cite[Theorem 6]{HK} and Theorem \ref{thm1}, $\mathbb A\setminus\{0\}=$ acc $\sigma (a\otimes b)\setminus\{0\}$
and $\mathbb B\setminus\{0\}=\mathbb L= I(a\otimes b)\setminus\{0\}$, respectively.
Recall in addition that,  according to \cite[Theorem 12(i)-(ii)]{Bo}, given a unital Banach algebra $A$ and $a\in A$,  necessary and sufficient for 
 $\sigma_{DR}(a)= \emptyset$ is that $\sigma (a)=\Pi (a)$.
In the following theorem, the relationship between $\sigma_{DR}(a\otimes b) $ and $\mathbb{D}= \sigma (a) \sigma_{DR}(b)\cup \sigma_{DR} (a)\sigma(b)$ will be studied.

\begin{thm} \label{thm2}Let $A$ and $B$ be two unital Banach algebras and consider $a\in A$ and $b\in B$.
Then, the following staements hold.\par
\noindent \rm (i)\it  If $\sigma (a)=\Pi (a)$ and $\sigma (b)=\Pi (b)$, then $\mathbb D=\emptyset =\sigma_{DR} (a\otimes b)$.\par
\noindent \rm (ii)\it If $\sigma_{DR}(a)\neq \emptyset$ or $\sigma_{DR}(b)\neq \emptyset$, then $\sigma_{DR}(a\otimes b)\setminus \{0\}=\mathbb D\setminus\{0\}$
 and $\sigma_{DR}(a\otimes b)\subseteq \mathbb D$.
 \end{thm}
\begin{proof} (i) Since the condition in the statement under consideration is equivalent to $\sigma_{DR}(a)=\emptyset$ and 
$\sigma_{DR}(b)=\emptyset$,  $\mathbb{D}=\emptyset$. On the other hand,  according to Theorem \ref{thm1}(i) and Theorem \ref{thm3}(iii),
 $\sigma (a\otimes b)=\Pi (a)\Pi (b)=$ 
iso $\sigma (a\otimes b)=\Pi (a\otimes b)$, which implies that $\sigma_{DR}(a\otimes b)=\emptyset$.\par

\noindent (ii) Suppose that $\sigma (a)=\pi (a)$ and $\sigma_{DR}(b)\neq \emptyset$. Then, 
$\mathbb D\setminus\{0\}=
(\Pi (a)\setminus\{0\})$ (acc $\sigma ( b)\setminus\{0\}) \cup (\Pi (a)\setminus\{0\}) (I(b)\setminus\{0\})=(\mathbb A\setminus\{0\})\cup (\mathbb B\setminus\{0\})
=\sigma_{DR}(a\otimes b) \setminus\{0\}$.\par

\indent  Note that if $0\notin\sigma (a\otimes b)$, then $\sigma_{DR}(a\otimes b)= \mathbb D$. As a result,
to prove that $\sigma_{DR}(a\otimes b)\subseteq \mathbb D$, it is enough to prove that if $0\in \sigma(a\otimes b)\setminus \mathbb D$,
then $0\notin \sigma_{DR}(a\otimes b)$. Now well, since $\sigma (a\otimes b)=\Pi (a)\Pi (b)\cup \mathbb D$, if $0\in \sigma (a\otimes b)\setminus \mathbb D$,
then it is not difficult to prove that $0\notin\Pi (a)$ and $0\in \Pi (b)$. However, according to Theorem \ref{thm3}(iii), 
$0\in\Pi (a\otimes b)$.\par

\indent The case $\sigma_{DR}(a)\neq \emptyset$ and $\sigma (b)=\Pi (b)$ can be proved interchanging $a$ with $b$.\par

\indent Next suppose that  $\sigma_{DR}(a)\neq \emptyset$ and $\sigma_{DR}(b)\neq \emptyset$. Using arguments similar
to the ones considered before, it is easy to prove that 
$\mathbb D\setminus\{0\}=(\mathbb A\setminus\{0\})\cup (\mathbb B\setminus\{0\})
=\sigma_{DR}(a\otimes b) \setminus\{0\}$ (\cite[Theorem 12(iv)]{Bo}). In addition, as before, 
 if $0\notin\sigma (a\otimes b)$, then $\sigma_{DR}(a\otimes b)= \mathbb D$.
On the other hand, if $0\in\sigma (a\otimes b)$, then $0\in\sigma (a)$
or $0\in\sigma (b)$. However, since $\sigma_{DR}(a)\neq \emptyset$ and $\sigma_{DR}(b)\neq \emptyset$,
$0\in \mathbb D$. Therefore, $\sigma_{DR}(a\otimes b)\subseteq \mathbb D$.
\end{proof}

\indent Under the same conditions of Theorem \ref{thm2}, note that $\sigma_{DR}(a\otimes b)$ and $\mathbb D$, in general, do not
coincide. In fact, if $a\in A$ is nilpotent (equivalently if $\sigma (a)=\Pi (a)=\{0\}$) and $b\in B$ is such that $\sigma_{DR}(b) \neq\emptyset$, then 
$\sigma_{DR}(a\otimes b)=\emptyset$ and $\mathbb D=\{0\}$. In fact, in this case $a\otimes b\in A\overline{\otimes} B$
is nilpotent, $\sigma (a\otimes b)=\Pi (a\otimes b )=\{0\}$, $\sigma_{DR} (a\otimes b)=\emptyset$ and $\mathbb D=\{0\}$.
Note, however, that if $a\in A$ is nilpotent and $\sigma_{DR}(b) =\emptyset$, then according to Theorem \ref{thm2}(i),
$\sigma_{DR}(a\otimes b)=\emptyset =\mathbb D$.
In addition, according to Theorem \ref{thm2}(ii),  $\sigma_{DR}(a\otimes b)\subsetneq \mathbb D$ if and only if $\mathbb D= \sigma_{DR}(a\otimes b) \cup\{0\}$,
$0\notin\sigma_{DR} (a\otimes b)$,
which in turn is equivalent to the fact that $0\in\mathbb D\cap\Pi (a\otimes b)$.
In the following theorems it will be characterized when  $\sigma_{DR}(a\otimes b)$ and $\mathbb D$ coincide. In first
place, the case in which only one Drazin spectrum is the empty set will be studied.
Naturally, nilpotent elements will  not be considered.
\par
\begin{thm} \label{thm4}Let $A$ and $B$ be two unital Banach algebras and consider $a\in A$ and $b\in B$.\par
\noindent \rm (i) \it If $\sigma (a)=\Pi (a)\neq \{0\}$ and $\sigma_{DR}(b)\neq\emptyset$, then necessary and sufficient for
$\sigma_{DR}(a\otimes b)$ and $\mathbb D$ to coincide is that $0\notin \Pi (a)$ or $0\notin\rho_{DR}(b)$.\par
\noindent \rm (ii)\it If $\sigma_{DR}(a )\neq\emptyset$ and $\sigma (b)=\Pi (b)\neq \{0\}$, then necessary and sufficient for
$\sigma_{DR}(a\otimes b)$ and $\mathbb D$ to coincide is that $0\notin\rho_{DR}(a)$ or $0\notin \Pi (b)$.
\end{thm}
\begin{proof} (i) Note that $\mathbb D=\Pi (a)\sigma_{DR} (b)$. Suppose that $\sigma_{DR}(a\otimes b)=\mathbb D$. Now well, if $0\in \Pi (a)$ and $0\in\rho_{DR}(b)$, then it
is clear that $0\in\mathbb D$ and, according to Theorem \ref{thm3}(iii), $0\in\Pi (a\otimes b)$. In particular,  $0\in\mathbb D\setminus\sigma_{DR} (a\otimes b)$, which is impossible. \par

\indent To prove the converse, suppose that $a$ is invertible or $b$ is not Drazin invertible. If  
$\sigma_{DR}(a\otimes b)$ and $\mathbb D$ does not coincide, then $0\in\mathbb D\setminus \sigma_{DR}(a\otimes b)$ thanks to Theorem \ref{thm2}(ii).
Now well, if $0\notin \sigma (a)=\Pi (a)$, then $0\in\sigma_{DR} (b)=$ acc $\sigma (b)\cup I(b)$.
Clearly, if $0\in$ acc $\sigma (b)$, then $0\in$ acc $\sigma (a\otimes b)\subseteq \sigma_{DR} (a\otimes b)$,
while if $0\in I(b)$, then according to Theorem \ref{thm3}(v), $0\in I(a\otimes b) \subseteq \sigma_{DR} (a\otimes b)$,
which  is impossible. On the other hand, suppose that $0\in \Pi (a)$ and $0\in \sigma_{DR} ( b)=$ acc $\sigma (b)\cup I(b)$.
If $0\in$ acc $\sigma (b)$, then since $\sigma (a)=\Pi (a)\neq \{0\}$, $0\in $ acc $\sigma (a\otimes b)\subseteq \sigma_{DR} (a\otimes b)$,
while if $0\in I(b)$, then according to Theorem \ref{thm3}(iv), $0\in I(a\otimes b)\subseteq \sigma_{DR} (a\otimes b)$, which is impossible.\par

\noindent (ii) Interchange $a$ with $b$ and use the same argument.
\end{proof}

\indent Note that to prove the following theorem, the description of the spectrum and the Drazin spectrum
considered in \cite[Theorem 12]{Bo} will be central.\par

\begin{thm} \label{thm5}Let $A$ and $B$ be two unital Banach algebras and consider $a\in A$ and $b\in B$ 
such $\sigma_{DR}(a )\neq\emptyset$ and $\sigma_{DR}(b )\neq\emptyset$.
Then, the following statements are equivalent.\par
\noindent \rm (i)\it $\mathbb D=\sigma_{DR}(a \otimes b)$;\par
\noindent \rm (ii)\it $a\in A$ and $b\in B$ are invertible or $a\otimes b$ is not Drazin invertible;\par
\noindent  \rm (iii)\it $a\in A$ and $b\in B$ are invertible or $0\notin ( \Pi (a)\cap \rho_{DR} (b) \cup  \rho_{DR}(a)\cap  \Pi (b))$.
\end{thm} 
\begin{proof}(i)$\Rightarrow$ (ii). If  $a\in A$ or $b\in B$ is not
invertible, then $0\in\sigma (a\otimes b)$. Now well, since $\sigma_{DR}(a )\neq\emptyset$ and $\sigma_{DR}(b )\neq\emptyset$,
$0\in \mathbb D=\sigma_{DR}(a \otimes b)$. In particular, $a\otimes b$ is not Drazin invertible. \par
\noindent (ii) $\Rightarrow$ (iii). Suppose that $0\in \sigma (a\otimes b)$. If $0\in \Pi (a)\cap \rho_{DR} (b)$
or  $0\in \rho_{DR} (a)\cap \Pi (b) $, then according to Theorem \ref{thm3}(iii), $0\in \Pi (a\otimes b)$, which is impossible.\par
\noindent (iii) $\Rightarrow$ (i). If $a\in A$ and $b\in B$ are invertible, then $0\notin\sigma (a\otimes b)$ and, 
according to Theorem \ref{thm2}(ii), $\mathbb D=\sigma_{DR}(a \otimes b)$. On the other hand, if
$0\in\sigma (a\otimes b)$, then $0\in\sigma (a)$ or $0\in\sigma (b)$, and as before, $0\in  \mathbb D$. Suppose that
$0\in\sigma (a)$ (the case $0\in\sigma (b)$ can be proved interchanging $a$ with $b$). To prove that $0\in\sigma_{DR}(a\otimes b)$, several cases
must be considered.\par

\indent If $0\in \Pi (a)$, then since $a$ is not nilpotent ($\sigma_{DR}(a)\neq\emptyset$), there is $\mu\in \sigma (a)\setminus \{0\}$. In
particular, if $0\in $ acc $\sigma (b)$, then $0\in$ acc $\sigma_{DR} (a\otimes b)\subseteq \sigma_{DR} (a\otimes b)$. In addition,
if $0\in I(b)$, according to Theorem \ref{thm3}(iv), $0\in I(a\otimes b)\subseteq\sigma_{DR} (a\otimes b)$.\par

\indent Next suppose that $0\in$ acc $\sigma (a)$. If $b$ is invertible or if $0\in $ acc $\sigma(b)$, then clearly
$0\in $ acc $\sigma_{DR} (a\otimes b)\subseteq \sigma_{DR} (a\otimes b)$. If $0\in \Pi (b)$,
then as before, since $b$ is not nilpotent, there exists $\nu\in\sigma (b)\setminus\{0\}$. 
In particular, $0\in $ acc $\sigma_{DR} (a\otimes b)\subseteq \sigma_{DR} (a\otimes b)$.
If $0\in I(b)$, then there are two possibilities. If there is $\nu\in\sigma (b)\setminus\{0\}$, then 
$0\in $ acc $\sigma_{DR} (a\otimes b)\subseteq \sigma_{DR} (a\otimes b)$, while if $\sigma (b)=I(b)=\{0\}$,
according to Theorem \ref{thm3}(ii), $0\in I(a\otimes b)\subseteq \sigma_{DR} (a\otimes b)$.\par

\indent Now consider the case $0\in I(a)$. As before, if $\sigma (a)=I(a)=\{0\}$, then since $b$ is not nilpotent,
$0\in I(a\otimes b)\subseteq \sigma_{DR}(a\otimes b)$. Concerning the case $\sigma(a)\setminus\{0\}\neq\emptyset$,
 if $0\in (\rho (b)\cup\Pi(b) \cup I(b))$, then according to Theorem \ref{thm3}(iv-vi), $0\in I(a\otimes b)\subseteq \sigma_{DR} (a\otimes b)$,
while if $0\in $ acc $\sigma(b)$,  $0\in$ acc $\sigma(a\otimes b)\subseteq \sigma_{DR} (a\otimes b)$.
\end{proof}

\section {\sfstp Elementary operators}\setcounter{df}{0}
\
\indent Recall that given Banach spaces $X$ and $Y$ and Banach space operator $S\in L(X)$ and $T\in L(Y)$,
the elementary operator defined by $S$ and $T$ is $\ST\colon L(Y,X)\to L(Y,X)$, $\ST (A)=SAT$, where $A\in L(Y,X)$
and $L(Y,X)$ stands for the Banach space of all bounded linear maps from $Y$ into $X$.
In this section the Drazin spectrum of  elementary operators between Banach spaces will be studied. To this end,
however, some preparation is needed. In first place, the definition of the axiomatic
tensor product of Banach spaces 
introduced  in \cite{E} will be recalled. 
\par
\pagestyle{myheadings} \markboth{  }{ \hskip7truecm \rm Enrico Boasso  }
\indent A pair $\langle X, \tilde{X}\rangle$ of Banach spaces will be called a dual 
pairing, if
$$
(A)\hbox{  }\tilde{X}=X^* \hbox{   or   }(B)\hbox{  }X=
\tilde{X}^*,
$$
where $X^*$ denotes the dual space of $X$. In both cases, the canonical bilinear mapping is denoted by
$$
X\times\tilde{X}\to\Bbb C ,\hbox{   }(x,u)\mapsto\langle x,u\rangle.
$$

\indent Given $\langle X,\tilde{X}\rangle$ a dual pairing, consider the 
subalgebra ${\mathcal L}(X)$ of $L(X)$ consisting of all 
operators $T\in L(X)$ for which there is an operator $T{'}
\in L(\tilde{X})$ with
$$
\langle Tx,u\rangle=\langle x,T{'}u\rangle,
$$
for all $x\in X$ and $u\in\tilde{X}$. It is clear that if the dual pairing 
is $\langle X,X^*\rangle$, then ${\mathcal L}(X)=L(X)$, and that if the 
dual pairing is $\langle X^*,X\rangle$, then ${\mathcal L}(X^*)=\{T^*\in L(X^*)\colon \hbox{ }T
\in  L(X)\}$, where $T^*\in L(X^*)$ denotes the adjoint map of
$T\in L(X)$. In particular, each operator of the form
$$
f_{y,v}\colon X\to X,\hbox{}x\mapsto\langle x,v\rangle y,
$$
is contained in ${\mathcal L(X)}$, for $y\in X$ and $v\in\tilde{X}$.
\par 

\indent Next the definition of the axiomatic tensor product given in \cite{E} will be recalled.\par

\begin{df}\label{def51}  
 Given two dual pairings $\langle X,\tilde{X}\rangle$ and 
$\langle Y,\tilde{Y}\rangle$, a tensor product of the Banach spaces $X$ and $Y$ 
relative to the dual pairings $\langle X,\tilde{X}\rangle$ and 
$\langle Y,\tilde{Y}\rangle$ is a Banach space $Z$ together with two
continuous bilinear mappings
$$
X\times Y\to Z,\hskip.5truecm(x,y)\mapsto x\otimes y;\hskip.5truecm {\mathcal L}(X)
\times{\mathcal L}(Y)\mapsto L(Z),\hskip.5truecm(T,S)\mapsto T\otimes S,
$$  
which satisfy the following conditions,

\noindent \rm (T1)\it $\hbox{  }\parallel x\otimes y\parallel=\parallel x\parallel
\parallel y\parallel$,\par
\noindent \rm (T2)\it $\hbox{  }T\otimes S(x\otimes y)=(Tx)\otimes(Sy)$,\par
\noindent \rm(T3)\it $\hbox{  }(T_1\otimes S_1)\circ (T_2\otimes S_2)=(T_1T_2)
\otimes(S_1S_2),\hbox{}I\otimes I=I$,\par
\noindent \rm (T4)\it $\hbox{  }R(f_{x,u}\otimes I)\subseteq\{x\otimes y\colon\hbox{}y
\in Y\},\hbox{}R(f_{y,v}\otimes I)\subseteq\{x\otimes y\colon\hbox{}x
\in X\}$.\par
\end{df}
\indent As in \cite{E}, $X\tilde{\otimes} Y$ will be written instead of $Z$. In addition, in \cite{E} two  
 main applications of Definition \ref{def51} were considered, namely, 
the completion $X\tilde{\otimes_{\alpha}}Y$ of the 
algebraic tensor product  of the
Banach spaces $X$ and $Y$ with respect to a quasi-uniform crossnorm 
$\alpha$, see \cite{I}, and an operator ideal between Banach spaces, which will be
the case considered in this section.\par

\begin{df}\label{def52} An operator ideal $J$ between Banach spaces $Y$ 
and $X$ is a linear subspace of $ L(Y,X)$ equipped 
with a 
space norm $\alpha$ such that\par
i)  $x\otimes y'\in J$ and $\alpha (x\otimes y')=
\parallel x
\parallel \parallel y'\parallel$,\par
ii) $SAT\in J$ and $\alpha (SAT)\le \parallel S\parallel \alpha (A) 
\parallel T\parallel$, \par

for $x\in X$, $y'\in Y'$, $A\in J$, 
$S\in L(X)$,
$T\in L(Y)$, and $x\otimes y'$ is the usual rank one
operator $Y\to X$, $y\mapsto <y,y'>x$.\par
\end{df}

\indent Examples of this kind of ideals were given in \cite{E,P}.\par

\indent Recall that, under the conditions of Definition \ref{def52}, an  operator ideal $J$ is naturally a tensor
product relative to $\langle X,X{'}\rangle$ and $\langle Y',
Y\rangle$, with the 
bilinear mappings
$$
X\times Y'\to J,\hbox{  } (x,y')\mapsto x\otimes y',
$$
$$
{\mathcal L}(X)\times{\mathcal L}(Y')\to  L(J),
\hbox{  } (S,T')\mapsto S\otimes T',
$$
where $S\otimes T' (A)=SAT$. In particular,
when $J=L(L(Y,X))$, $S\otimes T=\ST$, the elementary operator
defined by $S$ and $T$.\par 
\pagestyle{myheadings} \markboth{  }{ \hskip7truecm \rm Enrico Boasso  }
\indent Recall, in addition, that according to \cite[Corollaries 3.3-3.4]{E}, 
$\sigma (\ST)=\sigma (S)\sigma (T)$. Furthermore, note that if $Z$
is a tensor product relative to $\langle X, X^*\rangle$ and $\langle Y^*, Y\rangle$,
and if $M\subseteq X$ and $N\subseteq Y$ are closed and complemented subspaces
of $X$ and $Y$ respectively, then $P\otimes Q^* (Z)\subseteq Z$ is a tensor product
relative to $\langle M, M^*\rangle$ and $\langle N^*, N\rangle$, where $P^2=P\in L(X)$
and $Q^2=Q\in L(Y)$ are such that $R(P)=M$ and $R(Q)=N$  (actually, similar results
can be proved for all kinds of tensor product satisfying Definition \ref{def51}, however, since
the main objective of this section concerns operators ideals, this case has been focused on). 
On the other hand, if $W$ is a Banach space and $V\in L(W)$, then according to \cite[p. 139]{MM} and \cite[Theorems 3 and 12]{Bo},
$\Pi (V^*)=\Pi (V)$ and $\sigma_{DR} (V^*)=\sigma_{DR} (V)$, which implies that $I(V^*)=I(V)$ (naturally $\sigma (V^*)=\sigma (V)$,
acc $\sigma (V^*)=$ acc $\sigma (V)$ and iso $\sigma(V^*)=$ iso $\sigma (V)$). \par

\indent  Now well, it is not difficult to check that
arguments similars to the ones in the proofs of Theorems \ref{thm3}-\ref{thm1} can be used, considering in particular 
what has been recalled in the previous paragraph, to characterize 
the isolated points of $\sigma (S\otimes T)$, $S\otimes T\in L(Z)$,
$Z$ a tensor product relative to  dual pairings $\langle X,\tilde{X}\rangle$ and 
$\langle Y,\tilde{Y}\rangle$, $X$ and $Y$ Banach spaces and $S\in L(X)$, $T\in L(Y)$.
Therefore, the isolated points of $\sigma (\ST)$ can be fully described as it has been done
in section 2, where $\ST\in L(J)$, $J$ an operator ideal between Banach spaces $Y$ and $X$
that satisfies Definition \ref{def52}, $S\in L(X)$ and $T\in L(Y)$; this result naturally applies to $J=L(L(Y,X))$.
Furthermore, the arguments in Theorems \ref{thm2}-\ref{thm5} can be adapted to the objects satisfying
Definitions \ref{def51}-\ref{def52}. As a result, the Drazin spectrum of elementary operators is fully
characterized.  Note that in the following theorem, $\mathbb{D}= \sigma (S) \sigma_{DR}(T)\cup \sigma_{DR} (S)\sigma(T)$,
$S$ and $T$ as before.\par

\begin{thm} \label{thm53}Let $X$ and $Y$ be two  Banach spaces and consider $S\in L(X)$ and $T\in L(Y)$.
Then, the following statements hold.\par
\noindent \rm (i)\it  If $\sigma (S)=\Pi (S)$ and $\sigma (T)=\Pi (T)$, then $\mathbb D=\emptyset =\sigma_{DR} (\ST)$.\par
\noindent \rm (ii)\it If $\sigma_{DR}(S)\neq \emptyset$ or $\sigma_{DR}(T)\neq \emptyset$, then $\sigma_{DR}(\ST)\setminus \{0\}=\mathbb D\setminus\{0\}$,
$\sigma_{DR}(\ST)\subseteq \mathbb D$, and if $\sigma_{DR}(\ST)\subsetneq \mathbb D$, then 
$\mathbb D= \sigma_{DR}(\ST)\cup \{0\}$, $0\notin\sigma_{DR} (\ST)$.\par
\noindent \rm (iii) \it If $\sigma (S)=\Pi (S)\neq \{0\}$ and $\sigma_{DR}(T)\neq\emptyset$, then necessary and sufficient for
$\sigma_{DR}(\ST)$ and $\mathbb D$ to coincide is that $0\notin \Pi (S)$ or $0\notin\rho_{DR}(T)$.\par
\noindent \rm (iv)\it If $\sigma_{DR}(S )\neq\emptyset$ and $\sigma (T)=\Pi (T)\neq \{0\}$, then necessary and sufficient for
$\sigma_{DR}(\ST)$ and $\mathbb D$ to coincide is that $0\notin\rho_{DR}(S)$ or $0\notin \Pi (T)$.\par
\noindent \rm (v) \it Suppose that $\sigma_{DR}(S )\neq\emptyset$ and $\sigma_{DR}(T )\neq\emptyset$. Then, 
the following statements are equivalent.\par
\noindent \rm (va)\it $\mathbb D=\sigma_{DR}(\ST)$;\par
\noindent \rm (vb)\it $S$ and $T$ are invertible or $\ST$ is not Drazin invertible;\par
\noindent  \rm (vc)\it $S$ and $T$ are invertible or $0\notin ( \Pi (S)\cap \rho_{DR} (T) \cup  \rho_{DR}(S)\cap  \Pi (T))$.
 \end{thm}
\begin{proof}Adapt the proofs of Theorems  \ref{thm2}-\ref{thm5} to the case under consideration. Details are left to the reader.
\end{proof}

\bigskip

\noindent \normalsize \rm Enrico Boasso\par
  \noindent  E-mail: enrico\_odisseo@yahoo.it \par


\vskip.3truecm


\vskip.3truecm
\begin{thebibliography}{20}

\bibitem{AD} P. Aiena and B. P. Duggal, \emph{Tensor products, multiplications and Weyl's theorem},
Rend. Circ. Math. Palermo Ser. II 54 (2005), 387-395.

\bibitem{BS} M. Berkani and M. Sarih, \emph{An Atkinson-type theorem for B-fredholm operators}, 
Studia Math. 148 (2001), 251-257.

\bibitem{Bo} E. Boasso, \emph{Drazin spectra of Banach space operators and Banach
algebra elements}, J. Math. Anal. Appl. 359 (2009), 48-55.
\bibitem{BDJ} E. Boasso, B. P. Duggal and I. H. Jeon, 
\emph{Generalized Browder's and Weyl,s theorems for left and right
multiplication operators}, J. Math. Anal. Appl. 370 (2010),
461-471.

\bibitem{BP} A. Brown and C. Pearcy, \emph{ Spectra of tensor products of operators} Proc. Amer. Math. Soc. 17 (1966),
162-166.  

\bibitem{CD} R. E. Curto and A. T. Dash, \emph{Browder spectral systems}, Proc. Amer. Math. Soc. 103
(1988), 407-413.

\bibitem{D} M. P. Drazin, \emph{Pseudo-inverses in associative rings and semigroups}, Amer. Math. Monthly 65 
(1958), 506-514.

\bibitem{Du} B. P. Duggal, \emph{Browder-Weyl theorems, tensor products and multiplications}, J. Math.
Anal. Appl. 359 (2009), 631-636.


\bibitem{DDK} B. P. Duggal, S. V. Djordjevi\' c and C. S. Kubrusly, \emph{On the $a$-Browder and $a$-Weyl spectra of tensor
products}, Rend. Circ. Math. Palermo 59 (2010), 473-481.


\bibitem{DHK} B. P. Duggal, R. Harte and A. H. Kim, \emph{Weyl's theorem,
tensor products and multiplication operators II}, Glasg. Math. J.,
 52 (2010), 705-709.

\bibitem{E} J. Eschmeier, \emph{Tensor products and elementary
operators}, J. Reine Angew. Math. 390 (1988), 47-66.


\bibitem{H} R. Harte, \emph{Invertibility and Singularity for Bounded Linear operators},
Marcel Dekker Inc., New York and 
Basel, 1988. 

\bibitem{HK} R. E. Harte and A.  H. Kim,  \emph{Weyl's theorem,
tensor products and multiplication operators},  J. Math. Anal.
Appl. 336 (2007), 1124-1131.

\bibitem{I} T. Ichinose, \emph{Spectral properties of tensor 
products of linear operators. I}, Trans. Amer. Math. Soc.
 235 (1978), 75-113.

\bibitem{I2} T. Ichinose, \emph{Spectral properties of tensor 
products of linear operators. II: The approximate point spectrum and Kato essential spectrum}, Trans. Amer. Math. Soc.
 237 (1978), 223-254.

\bibitem{K} C. King,  \emph{A note on Drazin inverses}, Pacific J. Math. 70 (1977), 383-390.

\bibitem{Ko} J. J. Koliha, \emph{Isolated spectral points}, Proc. Amer. Math. Soc. 124 (1996), 3417-3424. 

\bibitem{KM} V. Kordula and V. M\" uller,  \emph{On the axiomatic theory of spectrum}, Studia Math. 119 
(1996), 109-128.

\bibitem{KD} C. S. Kubrusly and B. P.  Duggal, \emph{On Weyl and Browder spectra of tensor product}, Glasgow
Math. J., 50 (2008), 289-302

\bibitem{MM} M. Mbekhta and V. M\" uller, \emph{On the axiomatic theory of spectrum II}, Studia Math. 119 
(1996), 129-147.

\bibitem{P} A. Pietsch, \emph{Operator Ideals}, North-Holland Publishing Company, Amsterdam-New York-Oxford, 1980.

\bibitem{RS} S. Roch and B. Silbermann, \emph{Continuity of generalized inverses in Banach algebras}, Studia Math. 136
(1999), 197-227.

\bibitem{SK} Y.-H. Song and  A.-H. Kim, \emph{Weyl's theorem for tensor products}, Glasg. Math. J. 46 (2004), 301–304.
\pagestyle{myheadings} \markboth{  }{ \hskip7truecm \rm Enrico Boasso  }
\bibitem{T} A. E. Taylor, \emph{Introduction to Functional Analysis}, Wiley and Sons, New York, 1958.
\end{thebibliography}
\end{document}